\documentclass[a4paper,11pt,reqno]{amsart}

\usepackage[utf8]{inputenc}
\usepackage[english]{babel}
\usepackage{amsmath}
\usepackage{amssymb}
\usepackage{amsfonts}
\usepackage{amsthm}
\usepackage{mathrsfs}
\usepackage{graphicx}
\usepackage{enumitem}
\setlist[itemize]{leftmargin=*}
\usepackage{xcolor}
\usepackage{url}
\usepackage{bm}
\usepackage[all]{xy}
\usepackage{xcolor}
\usepackage{esint}

\usepackage{xcolor}
\usepackage{soul}
\usepackage{stmaryrd}

\usepackage{doi}
\usepackage{hyperref}
\hypersetup{
	colorlinks,
	linkcolor={blue!60!black},
	citecolor={green!60!black},
	urlcolor={green!60!black}
}
\usepackage[capitalise,nameinlink]{cleveref}
\crefformat{equation}{#2(#1)#3}
\crefrangeformat{equation}{#3(#1)#4 to~#5(#2)#6}
\crefmultiformat{equation}{#2(#1)#3}{ and~#2(#1)#3}{, #2(#1)#3}{ and~#2(#1)#3}
\usepackage{mathtools}
\usepackage{etoolbox}
\usepackage{chngcntr}
\usepackage[a4paper,margin=3cm]{geometry}

\newcommand{\R}{\mathbb{R}}

\newcommand{\de}{\partial}

\newcommand{\tr}{\text{trace}}
\renewcommand{\bar}{\overline}

\newcommand{\vol}{\mathrm{vol}}
\newcommand{\ang}[1]{\langle #1\rangle}

\renewcommand{\epsilon}{\varepsilon}

\theoremstyle{definition}
\newtheorem{definition}{Definition}
\newtheorem{rmk}[definition]{Remark}
\newtheorem*{definition*}{Definition}
\newtheorem*{rmk*}{Remark}
\newtheorem*{ack*}{Acknowledgement}
\newtheorem*{acks*}{Acknowledgements}
\newtheorem{thm}[definition]{Theorem}
\newtheorem{lemma}[definition]{Lemma}
\newtheorem{corollary}[definition]{Corollary}

\newtheorem*{thm*}{Theorem}

\newtheorem*{lemma*}{Lemma}
\newtheorem*{corollary*}{Corollary}
\newtheorem*{proposition*}{Proposition}
\newtheorem*{claim*}{Claim}
\newtheorem*{conj*}{Conjecture}

\numberwithin{equation}{section}
\numberwithin{definition}{section}

\crefname{thm}{Theorem}{Theorems}

\begin{document}

\title[new weighted inequalities]{New weighted inequalities on two--manifolds}
	\author[A. Halavati]{Aria Halavati}
	\address{Courant Institute of Mathematical Sciences, New York University, 251 Mercer Street, New York, NY 10012, United States of America.}
	\email{aria.halavati@cims.nyu.edu}

    \address{Bocconi University, Department of Decision Sciences, Via Guglielmo R\"ontgen 1, 20136 Milano,
Italy}
\email{aria.halavati@unibocconi.it}
  \dedicatory{Dedicated to the memory of Kian Pirfalak (2013--2022)}
	\begin{abstract}
We establish a new class of weighted $L^2$ Poincar\'e and elliptic functional inequalities on smooth two-manifolds with explicit constants, for a family of weights satisfying a differential equation. This family includes, in particular, weights comparable to products of positive powers of the geodesic distance to finitely many points. Our primary motivation is the derivation of estimates associated with a weighted Hodge decomposition for one-forms.
\end{abstract}

\maketitle
\frenchspacing

\section{Introduction}

\subsection{Introduction}

In this article we investigate \textit{multipolar} weighted elliptic and Poincar\'e-type inequalities on two-manifolds. Our main motivation comes from the \textit{weighted Hodge decomposition} of one-forms on a Riemannian two-manifold $M^2$ with boundary; see \cref{weighted-hodge}.

More specifically, we consider weights including, but not limited to, those of the form
\begin{align*}
    \omega(x) \sim \prod_{k=1}^{N} d_M(x,x_k)^{\alpha_k}.
\end{align*}
Here $\{x_k\}_{k=1}^N \subset M^2$ is a finite collection of points, $d_M(x,y)$ denotes the geodesic distance on $M^2$, and $\alpha_k>0$. In fact, the inequalities in this paper apply to any weight $\omega$ such that
\begin{align}
    \omega^2\Delta_g\log\omega = 0\,,
\end{align}
possibly with a controlled right-hand side.
This note, together with \cite{Halavati-stability,DHP}, forms the first part of a trilogy and develops the main analytical tools and estimates used in the proof of quantitative stability for Yang--Mills--Higgs instantons in \cite{Halavati-stability}. We present the results here in a more general framework, since we expect them to be useful in other contexts as well.

\subsection{Relation with previous work}

Weighted functional inequalities, such as the classical Hardy inequality, have long played an important role in analysis and geometry. For instance, the Caffarelli--Kohn--Nirenberg inequalities \cite{CKN} provide the basic unipolar model for first-order weighted interpolation inequalities in Euclidean space as follows:
\begin{align}
\begin{aligned}
    \bigl\||x|^\gamma u\bigr\|_{L^r(\mathbb R^n)}
    &\leq
    C
    \bigl\||x|^\alpha \nabla u\bigr\|_{L^p(\mathbb R^n)}^\theta
    \bigl\||x|^\beta u\bigr\|_{L^q(\mathbb R^n)}^{1-\theta},
    \\
    \frac1r+\frac{\gamma}{n}
    &=
    \theta\left(\frac1p+\frac{\alpha-1}{n}\right)
    +
    (1-\theta)\left(\frac1q+\frac{\beta}{n}\right).
\end{aligned}
\end{align}
These were used in the analysis of the Navier--Stokes equations in \cite{CKN-2} and provide examples of \textit{unipolar} weights, namely weights vanishing at a single point. Such inequalities, together with their optimal constants, were studied in detail in \cite{Catrina}. Later in \cite{Felli-Terracini} a special \textit{multipolar} version of these inequalities was studied in dimension $N\geq 3$, in which the authors prove:
\begin{align}
    \int_{\R^N} \sum_{i=1}^{k} \lambda_i \frac{|u|^2}{|x-a_i|^2} \leq \int_{\R^N}|\nabla u|^2\,,
\end{align}
provided that $\sum_{\substack{i=1,\dots,k\\ \lambda_i > 0}}\lambda_i < \frac{(N-2)^2}{4}$. A different class of weights was considered in \cite{Cazacu-Zuazua} in which they prove, for $N\geq 3$:
\begin{align}
    \int |\nabla u|^2\,dx
    \geq
    \frac{(N-2)^2}{n^2}
    \sum_{1\leq i<j\leq n}
    \int
    \frac{|x_i-x_j|^2}{|x-x_i|^2 |x-x_j|^2}
    u^2\,dx.
\end{align}
The Riemannian version was later established in \cite{multipolar}. Moreover in \cite{Cazacu} the author was able to prove a special case in dimension $N=2$ if the poles are located at the boundary of the support. In another work \cite{Canale} these techniques were generalized to prove inequalities with more general potentials. 

However, these results concern the classical inverse-square Hardy potential and are naturally formulated in dimensions \(N\geq 3\), where the standard Hardy constant is positive. By contrast, the two-dimensional situation is critical: the usual inverse-square Hardy inequality degenerates; consequently, we also require our weights to behave like the distance to these poles to any \textit{positive} power. This is also reflected in the identity $\omega^2\Delta_g\log\omega = -\kappa\omega^2$, where the point masses from $\Delta_g\log\omega$ would be suppressed by $\omega^2$ if $\omega$ vanishes at the poles. Thus the first main theorem of this paper can be viewed as a two-dimensional multipolar analogue of the CKN inequality, rather than as a direct instance of the higher-dimensional multipolar Hardy theory. In addition to $n=2$, the condition we impose is flexible enough to include any positive power of the distance to the multipoles, and it can be written intrinsically so as to apply to any smooth Riemannian two-manifold.

The elliptic estimates proved here are also different in nature from the usual Hardy--Rellich inequalities and the literature cited above. Classical Rellich-type estimates control lower-order weighted norms by second derivatives and have been developed in several settings, including Euclidean domains and Riemannian manifolds. In the present work we are able to prove elliptic estimates, as opposed to only interpolation inequalities requiring control of only the trace of the Hessian. The proof uses a two-dimensional cancellation for the trace-free tensor \(\omega^2\nabla^2\log\omega\) encoded in \cref{lemma-1}. This cancellation is not part of the standard Hardy--Rellich theory and is the reason the argument is intrinsically two-dimensional. One might also be tempted to use Muckenhoupt elliptic theory; however, the weights considered here are not in the relevant Muckenhoupt class (see \cite{Coifman-Fefferman}).

Finally, the elliptic estimates are motivated by the Hodge-theoretic question (see \cref{weighted-hodge}), and they should be compared with the broader literature on Hodge decompositions, for example \cite{nonlinear-hodge}. The result here is quantitative in a singular weighted norm: it compares the co-closed parts of the usual and weighted Hodge decompositions. This estimate is the form needed in the Yang--Mills--Higgs applications in \cite{Halavati-stability,DHP}, where the energy functional naturally contains a weighted factor. Another viewpoint is to regard this result as the Hodge decomposition on the infinite cylinder in log-polar coordinates.

Moreover, in this article we are able to prove these inequalities with universal constants that are independent of both the weight and the location of its vanishing points. This uniformity is one of the main technical advantages of our approach, and it is exploited in \cite{Halavati-stability} to strengthen the results obtained therein and to remove the dependence of the constant on the location of the poles.

\subsection{Motivation} To motivate our approach, we first discuss the weighted analogue of the classical Hodge decomposition. The standard Hodge decomposition of a one-form $A$ seeks, in a variational sense, the closest closed (or co-closed) form to $A$. In the weighted setting, consider the toy model with weight $|x|^2$ and the variational problem
\begin{align}\label{intro:weighted-hodge-toy}
    \inf_{\phi\in C^{\infty}_c(B_1^2)} \int_{B_1^2} |x|^2|A-\star d\phi|^2.
\end{align}
A natural function space for this problem is
\begin{align*}
    X = \left\{\phi\in C^\infty_c(B_1^2): \int_{B_1^2}|x|^2|d\phi|^2 < \infty \right\},
\end{align*}
equipped with the weighted inner product
\begin{align*}
    \langle \phi_1,\phi_2\rangle_X := \int_{B_1^2}|x|^2\langle d\phi_1,d\phi_2\rangle.
\end{align*}
Passing to the completion $\bar X$ with respect to the induced norm yields a natural framework for solving \eqref{intro:weighted-hodge-toy} by the direct method of the calculus of variations.

A key result, which follows from a special case of the Caffarelli--Kohn--Nirenberg interpolation inequalities \cite{CKN}, states that
\begin{align*}
    \int_{\R^2} |f|^2 \leq \int_{\R^2} |x|^2|df|^2
    \qquad \text{for all } f\in C_c^{\infty}(\R^2).
\end{align*}
As observed in \cite{Catrina}, this inequality admits a geometric interpretation via the log-polar transformation
\[
B_1^2 \ni x \longmapsto (-\log |x|,\theta)\in [0,\infty)\times S^1.
\]
Writing $u=|x|f$, the weighted Dirichlet term becomes the Sobolev norm on the infinite cylinder:
\begin{align*}
    \int_{\R^2}|x|^2 |df|^2
    =
    \int_{S^1\times[0,\infty)} \bigl(|u|^2 + |du|^2\bigr)\,d\vol_{S^1\times[0,\infty)}.
\end{align*}

Using weak lower semicontinuity in Sobolev spaces---either on the cylinder or directly through the CKN inequality---we obtain a minimizer $\phi$ of \eqref{intro:weighted-hodge-toy}. The associated Euler--Lagrange equation is
\begin{align*}
    d\bigl(|x|^2(A-\star d\phi)\bigr)=0,
\end{align*}
which implies that $A-\star d\phi$ is closed. Since it has vanishing trace, it follows that
\begin{align*}
    |x|A = |x|\star d\phi + |x|^{-1}d\xi
\end{align*}
for some compactly supported function $\xi$. This is the \textit{weighted Hodge decomposition}, and it satisfies the orthogonality relation
\begin{align*}
    \int_{B_1^2} |x|^2|A|^2
    =
    \int_{B_1^2}|x|^2|d\phi|^2
    +
    \int_{B_1^2}|x|^{-2}|d\xi|^2.
\end{align*}
Comparing the standard Hodge decomposition
\[
A = \star dp + dq
\]
with the weighted one, we obtain, among other things, the estimate
\begin{align*}
    \int_{B_1^2} |x|^{2+2\epsilon}|d(\phi-p)|^2
    \leq
    C\epsilon^{-2}\int_{B_1^2} |x|^{-2}|d\xi|^2.
\end{align*}
In fact, we show that the co-closed parts of the weighted and unweighted decompositions are close in $L^2$.

\subsection{General formulation and examples of weights}

In this article we extend the preceding heuristic to all weights satisfying the weak formulation \cref{weak-weight-equation} of the differential equation
\begin{equation}\label{strong-weight-equation}
    \omega^2 \Delta_g\log(\omega) = -\kappa(x)\omega^2,
\end{equation}
where $\omega$ is a positive weight in $W^{1,2}(M^2)$ and $\Delta_g$ is the Laplace--Beltrami operator on a smooth connected Riemannian two-manifold $(M^2,g)$, possibly with boundary.

This formulation allows us to handle weights vanishing at several points, while the proofs remain based on careful but elementary integration by parts and yield uniform constants.

We now mention a few examples.
\begin{itemize}
    \item Let $\Omega\subset\R^2$ be a bounded open set, and define
    \begin{align}\label{example-1}
        \omega(x)=\prod_{i=1}^{N}|x-x_i|^{\alpha_i},
    \end{align}
    where $x_1,\dots,x_N\in \Omega$ and $\alpha_1,\dots,\alpha_N>0$.

    \item Let $\Omega\subset \mathcal M^2$ be a bounded open domain in a smooth two-manifold, let $\mathcal G_p$ denote the Green's function of $\Omega$ with pole at $p$, and define
    \begin{align}\label{example-2}
        \omega(x)=\prod_{i=1}^{N} e^{-\alpha_i\mathcal G_{p_i}(x)},
    \end{align}
    where $p_1,\dots,p_N\in\Omega$ and $\alpha_1,\dots,\alpha_N>0$.
\end{itemize}
The weights in \cref{example-2} are comparable to multipolar distance weights:
\begin{align*}
    C^{-1}\prod_{i=1}^{N} d_M(x,p_i)^{\alpha_i}
    \leq
    \omega(x)
    \leq
    C\prod_{i=1}^{N} d_M(x,p_i)^{\alpha_i},
\end{align*}
where $d_M(x,y)$ denotes the geodesic distance on $M$.

\subsection{Main results}

Let $\Omega\subset \mathcal M^2$ be a smooth connected open domain, and let $\lambda_1$ denote the first Dirichlet eigenvalue of the Laplace--Beltrami operator on $\Omega$.

Our first result is a two-dimensional generalization of the Caffarelli--Kohn--Nirenberg interpolation inequality.

\begin{thm}\label{Generalized-CKN-thm}
Let $(\mathcal{M}^2,g)$ be a smooth two-manifold, let $\omega$ be a weight as in \cref{weak-weight-equation}, and let $\Omega\subset\mathcal{M}^2$ be a smooth open domain. Then every $f\in C_c^{\infty}(\Omega)$ satisfies
\begin{equation}\label{Generalized-CKN-inequality}
    \int_{\Omega} |\nabla\omega|^2 |f|^2\, d\vol_g
    \leq
    \int_{\Omega} \omega^2 |\nabla f|^2\, d\vol_g,
\end{equation}
provided that $\kappa \leq \lambda_1$.
\end{thm}

Our next theorem gives a homogeneous elliptic estimate.

\begin{thm}\label{elliptic-inequality-thm}
Let $(\mathcal{M}^2,g)$ be a smooth two-manifold, let $\omega$ be a weight as in \cref{weak-weight-equation}, and let $\Omega\subset\mathcal{M}^2$ be a smooth open domain. Then every $f\in C_c^{\infty}(\Omega)$ satisfies
\begin{equation}\label{elliptic-inequality}
    \int_{\Omega} \omega^2 |\nabla f|^2\, d\vol_g
    \leq
    \tau^{-1}\int_{\Omega}
    \left(
    2\frac{\omega^4}{|\nabla \omega|^2}|\Delta_g f|^2
    +
    5|\nabla \omega|^2 |f|^2
    \right)\, d\vol_g,
\end{equation}
provided that
\[
-\frac{\lambda_1}{8}(2-\tau)\leq \kappa \leq \lambda_1
\qquad\text{for some } 0<\tau\leq 2.
\]
\end{thm}

The next result is the main ingredient in the proof of \cref{weighted-hodge}. Here the homogeneity is broken in order to remove the term involving $|\nabla\omega|\,f$ from the right-hand side, at the expense of introducing a singular factor in $\epsilon$.

\begin{thm}\label{epsilon-elliptic-inequality-thm}
Let $(\mathcal{M}^2,g)$ be a smooth two-manifold, let $\omega$ be a weight as in \cref{weak-weight-equation} with $\kappa=0$, let $\epsilon>0$, and let $\Omega\subset\mathcal{M}^2$ be a smooth open domain. Then every $f\in C_c^{\infty}(\Omega)$ satisfies
\begin{equation}\label{epsilon-elliptic-inequality}
    \int_{\Omega} \omega^{2+2\epsilon} |\nabla f|^2\, d\vol_g
    \leq
    C\,\frac{(\sup_{\Omega}\omega)^{2\epsilon}}{\epsilon^2}
    \int_{\Omega} \frac{\omega^4}{|\nabla \omega|^2} |\Delta_g f|^2\, d\vol_g,
\end{equation}
where
\[
C \leq \frac{8\epsilon^2 + 5(1+\epsilon)^4}{8(1+\epsilon)^2},
\]
which tends to $\frac{5}{8}$ as $\epsilon\to 0$.
\end{thm}

We recall that the Laplace--Beltrami operator $\Delta_g$ acting on functions $u\in W^{1,2}(\mathcal M,g)$ is defined weakly by
\begin{align*}
    \int_{\Omega} -\Delta_g u\, v\, d\vol_g
    =
    \int_{\Omega} \langle \nabla u,\nabla v\rangle\, d\vol_g
    \qquad\text{for all } v\in W_0^{1,2}(\Omega).
\end{align*}

In the model case \(\omega=|x|^\alpha\), \cref{Generalized-CKN-thm,epsilon-elliptic-inequality-thm} yield the estimates
\begin{align*}
    \int_{\R^2}|x|^{2(\alpha-1)}|f|^2
    &\leq
    \alpha^{-2}\int_{\R^2}|x|^{2\alpha}|\nabla f|^2, \\
    \int_{\R^2}|x|^{2\alpha}|\nabla f|^2
    &\leq
    \alpha^{-2}\int_{\R^2}|x|^{2(\alpha+2)}|\Delta f|^2
    +
    \frac{5}{2}\alpha^2\int_{\R^2}|x|^{2(\alpha-1)}|f|^2, \\
    \int_{B_1}|x|^{2\alpha(1+\epsilon)}|\nabla f|^2
    &\leq
    C(\epsilon\alpha)^{-2}\int_{B_1}|x|^{2(\alpha+1)}|\Delta f|^2,
\end{align*}
provided that $\alpha>0$.

A central consequence of the preceding inequalities is the following estimate for the weighted Hodge decomposition.

\begin{corollary}\label{weighted-hodge}
Let $(\mathcal{M}^2,g)$ be a Riemannian two-manifold, let $\Omega\subset \mathcal{M}^2$ be a smooth open domain (with nonempty boundary), and let $\omega$ be a weight as in \cref{weak-weight-equation} with $\kappa=0$. Then every smooth compactly supported one-form $A\in C_c^{\infty}(\bigwedge^1\Omega)$ admits both a Hodge decomposition and a weighted Hodge decomposition of the form
\begin{align*}
    A = \star d\xi_1 + d\xi_2,
    \qquad
    \omega A = \star(\omega d\phi_1) + \omega^{-1}d\phi_2,
\end{align*}
for compactly supported functions $\xi_1,\xi_2,\phi_1,\phi_2$. Moreover, for every \(\epsilon>0\),
\begin{align*}
    \|\omega^{1+\epsilon} d(\xi_1-\phi_1)\|^2_{L^2(\Omega)}
    \leq
    C\frac{(\sup_{\Omega}\omega)^{2\epsilon}}{\epsilon^2}
    \|\omega^{-1} d\phi_2\|^2_{L^2(\Omega)}.
\end{align*}
\end{corollary}

The methods throughout the paper are inspired by \cite{CKN} and \cite{Catrina}. They are quite elementary and rely essentially only on Stokes' theorem. A crucial step in the proof, namely \cref{crucial-2}, uses \cref{lemma-1}, an identity for symmetric matrices in two dimensions that fails in higher dimensions.

\begin{rmk*}
In the case of unbounded domains, for example when $\mathcal M^2=\R^2$, we set $\lambda_1=0$ in \cref{elliptic-inequality-thm,Generalized-CKN-thm}.
\end{rmk*}

\begin{rmk*}
The conclusions of \cref{Generalized-CKN-thm,elliptic-inequality-thm} also hold for closed two-manifolds with $\Omega=\mathcal M^2$, provided one imposes the orthogonality condition
\[
\int_{\Omega}\omega f\, d\vol_g = 0.
\]
By contrast, \cref{epsilon-elliptic-inequality-thm} becomes trivial on closed manifolds when $\kappa=0$, since in that case
\[
\Delta_g(\omega^2)=4|d\omega|^2\geq 0,
\]
and hence \(\omega\) must be constant.
\end{rmk*}

\section{Proofs}
    \begin{definition}\label{weak-weight-equation}
      A weight $\omega\in W^{1,2}(\mathcal M^2)$ is admissible if it satisfies the following weak formulation of \cref{strong-weight-equation}: for every test function $\phi\in C_c^\infty(\mathcal M^2)$, one has
    \begin{equation*}
      \int_{\Omega} (4|\nabla \omega|^2 - 2\kappa\omega^2) \phi -  \omega^2\Delta_g\phi \;d\vol_g = 0\,.
    \end{equation*}
    \end{definition}
    To prove \cref{Generalized-CKN-thm,elliptic-inequality-thm,epsilon-elliptic-inequality-thm}, we use Stokes' theorem to relate the integral of a carefully chosen positive term to the difference between the right- and left-hand sides of \cref{elliptic-inequality,Generalized-CKN-inequality,epsilon-elliptic-inequality}. 
    \begin{rmk*}
        Before proceeding to the proofs, define the zero set $S:=\omega^{-1}(0)$. Then we may approximate any smooth function with finite weighted integral by a sequence of functions $f_k\in C_{c}^\infty(\Omega\setminus S)$ via cut-offs $\phi(\omega)$. If the integrals on the right-hand side are infinite, then the inequality is trivial.
    \end{rmk*}
    \begin{proof}[Proof of \cref{Generalized-CKN-thm}]
      We begin with the identity below:
      \begin{align*}
        0 \leq &\int_{\Omega} |\nabla (\omega f)|^2 \; d\vol_g =\int_{\Omega} |\omega \nabla f + \nabla\omega f|^2 \;d\vol_g\\
        = &\int_{\Omega} \omega^2 |\nabla f|^2 + |\nabla \omega|^2 |f|^2 + 2\ang{\omega \nabla \omega  ,\nabla f f } \;d\vol_g\,.
      \end{align*}
      After integrating by parts in the cross term and using \cref{weak-weight-equation}, we see that
      \begin{equation*}
        \int_{\Omega} 2\ang{\omega \nabla \omega  ,\nabla f f } \;d\vol_g = \int_{\Omega} -\frac{\omega^2}{2}\Delta_g(f^2) \;d\vol_g = \int_{\Omega} (\kappa\omega^2-2|\nabla \omega|^2) |f|^2 \;d\vol_g\,.
      \end{equation*}
      We then use $\kappa \leq \lambda_1$ to estimate
      \begin{equation*}
        \int_{\Omega} \kappa\omega^2|f|^2 \;d\vol_g \leq \int_{\Omega} |\nabla (\omega f)|^2 \;d\vol_g\,.
      \end{equation*}
      Finally, we conclude that
      \begin{equation*}
        0\leq \int_{\Omega} \omega^2|\nabla f|^2 - |\nabla \omega|^2 |f|^2 \;d\vol_g\,.
      \end{equation*}
    \end{proof}
    \begin{proof}[Proof of \cref{elliptic-inequality-thm}]
      Similarly, we begin by integrating a positive term:
      \begin{equation*}
        0\leq \int_{\Omega} \big|\frac{\omega^2}{|\nabla \omega|}\Delta_g f +  |\nabla \omega|f\big|^2 \;d\vol_g = \int_{\Omega} \frac{\omega^4}{|\nabla \omega|^2}|\Delta_g f|^2 + 2\omega^2 f \Delta_g f +  |\nabla \omega|^2|f|^2 \;d\vol_g\,.
      \end{equation*}
      By Stokes' theorem for the cross term and \cref{weak-weight-equation}, we get
      \begin{equation*}
        \int_{\Omega} 2\omega^2 f \Delta_g f \;d\vol_g = \int_{\Omega} -2\omega^2|\nabla f|^2 + (4|\nabla \omega|^2-2\kappa\omega^2) |f|^2 \;d\vol_g\,.
      \end{equation*}
      Since the assumption for an unbounded domain is $\kappa=0$, the proof follows immediately. 
      Otherwise, by the assumption $-\kappa \leq \lambda_1(\frac14-\frac\tau8)$, we see that
      \begin{equation*}
        \int_{\Omega} -2\kappa|\omega f|^2 \;d\vol_g  \leq \lambda_1(\frac12 - \frac\tau4)\int_{\Omega}|\omega f|^2 \;d\vol_g\,.
      \end{equation*}
      By the characterization of the first eigenvalue of the Laplace--Beltrami operator $\Delta_g$, we see that
      \begin{equation*}
        \lambda_1(\frac12-\frac\tau4)\int_{\Omega} \omega^2|f|^2 \;d\vol_g \leq (\frac12-\frac\tau4)\int_{\Omega} |\nabla (\omega f)|^2 \;d\vol_g\,.
      \end{equation*}
    Since $\kappa \leq \lambda_1$\,, \cref{Generalized-CKN-thm} applies and we get
      \begin{equation*}
        (\frac12-\frac\tau4)\int_{\Omega} |\nabla (\omega f)|^2 \;d\vol_g \leq (2-\tau)\int_{\Omega} \omega^2|\nabla f|^2 \;d\vol_g\,.
      \end{equation*}
      Finally, putting the estimates together, we conclude that
      \begin{equation*}
        0 \leq \int_{\Omega} 2\frac{\omega^4}{|\nabla \omega|^2}|\Delta_g f|^2 + 5|\nabla \omega|^2 |f|^2 - \tau\omega^2 |\nabla f|^2 d\vol_g\,.
      \end{equation*}
    \end{proof}
    In the proof of \cref{epsilon-elliptic-inequality-thm}, we deal with the weighted Hessian matrix $\omega^2\nabla^2\log(\omega)$. Since $\kappa=0$ here, \cref{strong-weight-equation} implies that this is a two-dimensional symmetric trace-free matrix. The following lemma uses this structure and is essential in the proof of \cref{epsilon-elliptic-inequality-thm}:
    \begin{lemma}\label{lemma-1}
      Let $A \in \R^{2\times2}$ be a symmetric matrix. Then, for any two real vectors $b,c\in \R^2$,
      \begin{equation}
        2(A:b\otimes c)\ang{b,c} - (A:b\otimes b)|c|^2 - (A:c\otimes c)|b|^2
        = -\tr(A)\ang{b,c^{\perp}}^2\,.
      \end{equation}
      Here $:$ denotes the Frobenius matrix product and $c^\perp$ is a vector perpendicular to $c$.
      \begin{proof}
        We first calculate the expression above in dimension $n$. Since $A$ is symmetric, it admits an orthonormal eigenbasis $e_i$ with real eigenvalues $\mu_i$. Setting $b_i = \ang{b,e_i}$ and $c_i = \ang{c,e_i}$, we compute
        \begin{align*}
          &2(A:b\otimes c)\ang{b,c} - (A:b\otimes b)|c|^2 - (A:c\otimes c)|b|^2\\
          &\qquad= -\sum_{1\leq i,j\leq n} \mu_i(b_ic_j-c_ib_j)^2\,.
        \end{align*}
        In the case $n=2$, this gives
        \begin{equation*}
          2(A:b\otimes c)\ang{b,c} - (A:b\otimes b)|c|^2 - (A:c\otimes c)|b|^2
          = -\tr(A)(b_1c_2-c_1b_2)^2\,.
        \end{equation*}
      \end{proof}
    \end{lemma}
    \begin{proof}[Proof of \cref{epsilon-elliptic-inequality-thm}]
      First, we integrate a carefully chosen positive term of the form below:
      \begin{align}
        0\leq &\int_{\Omega} \big|\frac{\omega^2}{|\nabla \omega|}\Delta_g f + 2\omega \ang{\frac{\nabla \omega}{|\nabla \omega|},\nabla f} + 2|\nabla \omega| f\big|^2 \;d\vol_g \nonumber\\
        =\;&\int_{\Omega} \frac{\omega^4}{|\nabla \omega|^2} |\Delta_g f|^2 + 4\omega^2\ang{\frac{d\omega}{|\nabla \omega|},\nabla f}^2 + 4 |\nabla \omega|^2 |f|^2 \label{squares}\\
        +\;& 4 \frac{\omega^3}{|\nabla \omega|^2} \ang{\nabla \omega,\nabla f} \Delta_g f + 4\omega^2 \Delta_g f f + 8 \ang{\omega \nabla \omega , f \nabla f } \;d\vol_g \,.\label{cross-terms}
      \end{align}
      For the first cross term in \cref{cross-terms}, Stokes' theorem, \cref{strong-weight-equation} (through the weak formulation in \cref{weak-weight-equation}), and the assumption $\kappa=0$ give
      \begin{align}
        &\int_{\Omega} 4 \frac{\omega^3}{|\nabla \omega|^2} \ang{\nabla \omega,\nabla f} \Delta_g f \;d\vol_g\nonumber\\
        =&\int_{\Omega} 2 \textrm{div}_g(\frac{\omega^3}{|\nabla \omega|^2}d\omega) |\nabla f|^2 - 4\nabla (\frac{\omega^3}{|\nabla \omega|^2} \nabla \omega): \nabla f\otimes \nabla f \;d\vol_g \nonumber\\
        = &\int_{\Omega} (4\omega^2 - 4\frac{\omega^4}{|\nabla\omega|^4}\nabla^2(\log(\omega)):\nabla\omega\otimes\nabla\omega)|\nabla f|^2 \label{temp-1}\\& - 4\ang{\nabla (\frac{\omega^3}{|\nabla \omega|^2} \nabla \omega): \nabla f\otimes \nabla f} \;d\vol_g \nonumber\,.
      \end{align}
      The last line follows from:
      \begin{align*}
          2\textrm{div}_g(\frac{\omega^3}{|\nabla\omega|^2}d\omega) &= 2\frac{\omega^3}{|\nabla\omega|^2}\Delta\omega + 6\omega^2 - 4 \frac{\omega^3}{|\nabla\omega|^4}\nabla^2\omega:\nabla\omega\otimes\nabla\omega\\
          &= 4\omega^2 + 2\frac{\omega^4}{|\nabla\omega|^2}\Delta(\log(\omega)) - 4\frac{\omega^4}{|\nabla\omega|^4}\ang{\nabla^2\log(\omega):\nabla\omega\otimes\nabla\omega}\\
          &= 4\omega^2 - 4\frac{\omega^4}{|\nabla\omega|^4}\ang{\nabla^2\log(\omega):\nabla\omega\otimes\nabla\omega} \,.
      \end{align*}
      Here we used the following identity:
      \begin{equation*}
        \omega \nabla^2 \omega = \omega^2 \nabla^2 \log(\omega) + \nabla\omega\otimes\nabla\omega\,,
      \end{equation*}
      for the second and third terms in \cref{temp-1}. We get
      \begin{align}
      \begin{aligned}\label{crucial-2}
        &-4\ang{\nabla (\frac{\omega^3}{|\nabla \omega|^2} \nabla \omega): \nabla f\otimes \nabla f} - 4\frac{\omega^4}{|\nabla\omega|^4}\ang{\nabla^2\log(\omega):\nabla\omega\otimes\nabla\omega}|\nabla f|^2 \\ 
        =&-8\omega^2\ang{\nabla f,\frac{\nabla \omega}{|\nabla \omega|}}^2 + 4\frac{\omega^4}{|\nabla \omega|^4}\big[2\ang{\nabla ^2\log(\omega):\nabla \omega \otimes \nabla f}\ang{\nabla \omega,\nabla f}\\
        -&\ang{\nabla^2\log(\omega):\nabla f\otimes \nabla f}|\nabla \omega|^2-\ang{\nabla^2\log(\omega):\nabla \omega\otimes \nabla \omega}|\nabla f|^2\big]\,.
        \end{aligned}
      \end{align}
      We apply \cref{lemma-1} with:
      \begin{equation*}
        A = \omega^2 \nabla^2 \log(\omega)\,,\;\;b = \frac{\nabla \omega}{|\nabla \omega|}\;\text{ and } c = \nabla f\,,
      \end{equation*}
      and $\tr(A) = \omega^2\Delta_g\log(\omega) = 0$ to see that
      \begin{align*}
        -4\ang{\nabla(\frac{\omega^3}{|\nabla \omega|^2} \nabla \omega): \nabla f\otimes \nabla f} - 4\frac{\omega^4}{|\nabla\omega|^4}\ang{\nabla^2(\log(\omega))&:\nabla\omega\otimes\nabla\omega}|\nabla f|^2 \\ &= -8\omega^2\ang{\nabla f,\frac{\nabla \omega}{|\nabla \omega|}}^2\,.
      \end{align*}
    For the second and third cross terms in \cref{cross-terms}, we see that
      \begin{equation*}
        \int_{\Omega} 4\omega^2 \Delta_g f f + 8 \ang{\omega \nabla \omega , f \nabla f } \;d\vol_g = \int_{\Omega} -4\omega^2 |\nabla f|^2 \;d\vol_g\,.
      \end{equation*}
       Then, putting the estimates together, we see that
      \begin{equation}
        4\int_{\Omega}\omega^2 \ang{\nabla f,\frac{\nabla \omega}{|\nabla \omega|}} ^2 - |\nabla \omega|^2|f|^2\;d\vol_g \leq \int_{\Omega} \frac{\omega^4}{|\nabla \omega|^2} |\Delta_g f|^2\;d\vol_g\label{prelim-bound}\,.
      \end{equation}
      Using \cref{strong-weight-equation} with $\kappa=0$, we get the following identity for \cref{prelim-bound}:
      \begin{align}
        &\int_{\Omega}\omega^2 \ang{\nabla f,\frac{\nabla \omega}{|\nabla \omega|}} ^2 - |\nabla \omega|^2|f|^2\;d\vol_g \nonumber\\ = &\int_{\Omega} |\omega \ang{\nabla f,\frac{\nabla \omega}{|\nabla \omega|}} + |\nabla \omega| f |^2\;d\vol_g\nonumber\\\geq &(\sup_{\Omega} \omega)^{-2\epsilon} \int_{\Omega} \omega^{2\epsilon}|\omega \ang{\nabla f,\frac{\nabla \omega}{|\nabla \omega|}} + |\nabla \omega| f |^2\;d\vol_g\label{temp-2}\,.
      \end{align}
    Notice that $\omega^{1+\epsilon}$ also satisfies \cref{strong-weight-equation} weakly when $\kappa=0$, so we compute \cref{temp-2} as follows:
      \begin{align}
        &\int_{\Omega} \omega^{2\epsilon}|\omega \ang{\nabla f,\frac{\nabla \omega}{|\nabla \omega|}} + |\nabla \omega| f |^2\;d\vol_g \nonumber\\
        =&\int_{\Omega} \omega^{2+2\epsilon} \ang{\nabla f,\frac{\nabla \omega}{|\nabla \omega|}}^2 + \omega^{2\epsilon}|\nabla \omega|^2 |f|^2 + 2\omega^{1+2\epsilon} \ang{\nabla \omega,\nabla f} f \;d\vol_g \nonumber \\
        =&\int_{\Omega} \omega^{2+2\epsilon} \ang{\nabla f,\frac{\nabla \omega}{|\nabla \omega|}}^2 + \omega^{2\epsilon}|\nabla \omega|^2 |f|^2 -\Delta_g(\frac{\omega^{2+2\epsilon}}{2+2\epsilon}) |f|^2 \;d\vol_g \nonumber \\
        =&\int_{\Omega} \omega^{2+2\epsilon}\ang{\nabla f,\frac{\nabla \omega}{|\nabla \omega|}}^2 -(1+2\epsilon) \omega^{2\epsilon}|\nabla \omega|^2 |f|^2 \;d\vol_g\label{temp-3}\,.
      \end{align}
    Notice that for $\omega^{1+\epsilon}$, we have
      \begin{align*}
        0 \leq \int_{\Omega}& \omega^{2\epsilon}| \omega\ang{\nabla f,\frac{\nabla \omega}{|\nabla \omega|}} + (1+\epsilon) |\nabla \omega| f |^2 \;d\vol_g\\
        =\int_{\Omega}& \omega^{2+2\epsilon}\ang{\nabla f,\frac{\nabla \omega}{|\nabla \omega|}}^2 + (1+\epsilon)^2 \omega^{2\epsilon}|\nabla \omega|^2 |f|^2 \\ &\;\;\;\;\;+2(1+\epsilon) \omega^{1+2\epsilon}\ang{\nabla \omega,\nabla f}f \;d\vol_g\\
        =\int_{\Omega}& \omega^{2+2\epsilon}\ang{\nabla f,\frac{\nabla \omega}{|\nabla \omega|}}^2 - (1+\epsilon)^2 \omega^{2\epsilon}|\nabla \omega|^2 |f|^2 \; d\vol_g\,.
      \end{align*}
    We expand the square $(1+\epsilon)^2$ to get a lower bound for \cref{temp-3}:
      \begin{align*}
        \int_{\Omega} \omega^{2+2\epsilon}\ang{\nabla f,\frac{\nabla \omega}{|\nabla \omega|}}^2 -(1+2\epsilon) \omega^{2\epsilon}|\nabla \omega|^2 |f|^2 \;d\vol_g \geq \epsilon^2 \int_{\Omega} \omega^{2\epsilon}|\nabla \omega|^2|f|^2 \;d\vol_g\,.
      \end{align*}
      Thus we obtain the preliminary inequality
      \begin{equation}
        \int_{\Omega} \omega^{2\epsilon}|\nabla \omega|^2 |f|^2 \;d\vol_g\leq \frac{(\sup_{\Omega} \omega)^{2\epsilon} }{4\epsilon^2}\int_{\Omega} \frac{|\omega|^4}{|\nabla \omega|^2}|\Delta_g f|^2 \;d\vol_g\label{final-eq-1}\,.
      \end{equation}
      Then we use \cref{elliptic-inequality-thm} for $\omega^{1+\epsilon}$ with $\kappa=0$ and $\tau = 2$ to see that
      \begin{align*}
        \int_{\Omega}2\omega^{2+2\epsilon} &|\nabla f|^2 \;d\vol_g \\ &\leq \int_{\Omega} 2\frac{\omega^{4+2\epsilon}}{(1+\epsilon)^2|\nabla \omega|^2}|\Delta_g f|^2 + 5 (1+\epsilon)^2\omega^{2\epsilon}|\nabla \omega|^2|f|^2\;d\vol_g\,.
      \end{align*}
    Finally, we use \cref{final-eq-1} to conclude that
      \begin{align*}
        \int_{\Omega}\omega^{2+2\epsilon} |\nabla f|^2 \;d\vol_g \leq (\frac{8\epsilon^2 + 5(1+\epsilon)^4}{8(1+\epsilon)^2})\frac{(\sup_{\Omega}\omega)^{2\epsilon}}{\epsilon^2}\int_{\Omega} \frac{\omega^{4}}{|\nabla \omega|^2}|\Delta_g f|^2 \;d\vol_g\,.
      \end{align*}
    \end{proof}
    \begin{rmk}
      In the case $\mathcal{M}^2 = B^2_1(0)\subset \R^2$ and $\omega=|x|$, after the log-polar transformation $B^2_1 \rightarrow \R^{+} \times S^1 = \mathcal{C}$ given by $t = -\log(|x|)$ and $\theta = \arctan(\frac{y}{x})$, or equivalently after a conformal change of metric with factor $\frac{1}{|x|^2}$, and defining $f = |x|^{-1}u$ for $f \in C_c^{\infty}(B^2_1(0))$, we see that
      \begin{align}
        \int_{B^2_1(0)} |\nabla \omega|^2|f|^2 &= \int_{\mathcal{C}} |u|^2 d\vol_{\mathcal{C}}\label{temp-5}\,,\\
        \int_{B^2_1(0)} \omega^2|\nabla f|^2 &= \int_{\mathcal{C}} |\nabla u|^2 + |u|^2 d\vol_{\mathcal{C}}\nonumber\,,\\
        \int_{B^2_1(0)} \frac{\omega^4}{|\nabla \omega|^2}|\Delta f|^2 &= \int_{\mathcal{C}} |\Delta u + 2 \de_t u + u|^2 d\vol_{\mathcal{C}}\label{temp-4}\,.
      \end{align}
      After squaring and integrating by parts, \cref{temp-4} becomes
      \begin{equation*}
        \int_{B^2_1(0)} \frac{\omega^4}{|\nabla \omega|^2}|\Delta f|^2 = \int_{\mathcal{C}} |\de_{tt} u|^2 + |\de_{t\theta} u|^2 + 2|\de_t u|^2 + |\de_{\theta\theta} u + u|^2\,.
      \end{equation*}
      If $u(t,\theta) = \sin(\theta)$, then \cref{temp-4} vanishes; however, \cref{temp-5} does not vanish, so the term $|\nabla \omega|f$ on the right-hand side of \cref{elliptic-inequality} is necessary. However, the extra $\epsilon$ in the power
      \begin{equation*}
        \int_{B^2_1(0)} \omega^{2+2\epsilon}|\nabla f|^2 = \int_{\mathcal{C}} (|\nabla u|^2 + |u|^2)e^{-2\epsilon t} d\vol_\mathcal{C}\,,
      \end{equation*}
      compactifies the domain $\R^{+}\times S^1$ with a total measure of $\epsilon^{-2}$. This provides some insight into \cref{epsilon-elliptic-inequality-thm} and the constants in \cref{epsilon-elliptic-inequality}.
    \end{rmk}
    We conclude the paper with the proof of the weighted Hodge decomposition estimates:
      \begin{proof}[Proof of \cref{weighted-hodge}]
        We consider the two variational problems below:
        \begin{align}\label{variational-hodge}
          \inf_{\xi \in C^{\infty}_c(\Omega)} \int_{\Omega} |A - \star d\xi|^2 \;d\vol_g\quad\text{and}\quad
          \inf_{\phi\in C^{\infty}_c(\Omega)} \int_{\Omega} \omega^2|A - \star d\phi|^2 \;d\vol_g\,.
        \end{align}
        Let $W^{1,2}_0(\omega^2,\Omega)$ be the completion of $C^\infty_c(\Omega)$ under the $\omega^2$-weighted norm
        \begin{align*}
            \|u\|^2_{W^{1,2}_0(\omega^2,\Omega)} := \int_{\Omega}\omega^2(|u|^2 + |du|^2)\,.
        \end{align*}
        The existence of minimizers of \cref{variational-hodge} follows from the direct method in the calculus of variations. The Euler--Lagrange equations for minimizers imply that
        \begin{align*}
          \star d(A - \star d\xi_1) = 0 &\Rightarrow \text{ there exists } \xi_2 \text{ such that } A-\star d\xi_1 = d\xi_2\text{ and}\\
          \star d(\omega^2(A - \star d\phi_1 )) = 0 &\Rightarrow \text{ there exists } \phi_2 \text{ such that }\omega^2(A - \star d\phi_1) = d\phi_2 \,.
        \end{align*}
        These identities hold in the sense of distributions, where $\xi_2,\phi_2$ are both compactly supported and $\xi_1,\xi_2\in W^{1,2}_0(\Omega)$, $\phi_1 \in W^{1,2}(\omega^2,\Omega)$, $d\phi_2 \in W^{1,2}_0(\Omega)$, and $\omega^{-1}d\phi_2\in L^2(\Omega)$ by the orthogonality below:
        \begin{align}
            \int_{\Omega}\omega^2|A|^2 = \int_\Omega \omega^2|d\phi_1|^2 + \omega^{-2}|d\phi_2|^2\,.
        \end{align}
        Then, by a direct application of \cref{epsilon-elliptic-inequality-thm},
        \begin{align*}
          \|\omega^{1+\epsilon} d(\xi_1 - \phi_1 )\|^2_{L^2(\Omega)} \leq C\frac{(\sup_{\Omega}\omega)^{2\epsilon}}{\epsilon^2} \|\frac{\omega^2}{|d\omega|} \Delta_g(\xi_1 - \phi_1 )\|^2_{L^2(\Omega)}.
        \end{align*}
        Moreover,
        \begin{equation*}
          \|\frac{\omega^2}{|d\omega|} \Delta_g(\xi_1 - \phi_1 )\|^2_{L^2(\Omega)} = \|\frac{\omega^2}{|d\omega|} d(\omega^{-2}d\phi_2 - d\xi_2)\|^2_{L^2(\Omega)} \leq 4\|\omega^{-1} d\phi_2 \|^2_{L^2(\Omega)}\,.
        \end{equation*}
        Combining these two estimates concludes the proof.
      \end{proof}
      
    \begin{ack*}
      I would like to thank Guido de Philippis and Alessandro Pigati for their support and mentorship, and Robert Kohn for his interest and related discussions. I would also like to express my gratitude to the anonymous referee, whose comments greatly improved the presentation of the manuscript. The author has been partially supported by NSF grant DMS-2055686, the Simons Foundation, and the European Union through the European Research Council (ERC), StG ``ANGEVA'', project number 101076411.

The views and opinions expressed are, however, those of the author only and do not necessarily reflect those of the European Union or the European Research Council. Neither the European Union nor the granting authority can be held responsible for them.
      
      \textit{Agradezco a J.A. por su amor y apoyo.}
    \end{ack*}
    \bibliographystyle{acm}
    \bibliography{references}
\end{document}